\documentclass[a4paper, 10pt,]{article}
\usepackage[utf8]{inputenc}
\usepackage{bigints}
\usepackage[english]{babel}
\usepackage[ntheorem]{empheq}
\usepackage{graphicx} 
\usepackage{tikz}
\usepackage{bbm}

\usetikzlibrary{cd}
\usepackage{pgfplots}

\usepackage{amsthm, amsfonts, amssymb, amscd}
\usepackage{caption}
\usepackage{subcaption}
\usepackage{mdframed}
\usepackage{lipsum}
\usepackage{tikz}
\usepackage{enumerate}
\usepackage{makeidx}
\usetikzlibrary{arrows}
\usetikzlibrary{calc}
\usepackage{graphicx}
\usepackage[hyphens]{url}

\def\R{\mathbb{R}}

\def\s{\mathcal{S}}

\def\pik{\pi_{\ker(L)}}

\DeclareMathOperator{\dls}{\deg}
\DeclareMathOperator{\db}{\deg}
\DeclareMathOperator{\id}{id}

\DeclareMathOperator{\fp}{fp}

\newtheorem{theorem}{Theorem}[section]
\newtheorem{proposition}[theorem]{Proposition}  
\newtheorem{definition}[theorem]{Definition} 
\newtheorem{obs}[theorem]{Remark}
 \newtheorem{lemma}[theorem]{Lemma}
 
\newtheorem{corollary}{Corollary}[theorem]

\newcommand\restr[2]{{
  \left.\kern-\nulldelimiterspace 
  #1 
  \vphantom{\big|} 
  \right|_{#2} 
  }}
 \title{On an affinity principle by Krasnoselskii}  
 \author{P. Amster and J. Epstein} 
 \date{}

\begin{document}

\maketitle

	\begin{center}
 Departamento de Matem\'atica\\ Facultad de Ciencias Exactas y Naturales, Universidad de Buenos Aires\\ \& IMAS-CONICET\\ 
		Ciudad Universitaria - Pabell\'on I, 1428, Buenos Aires, Argentina
		\vspace{0.2cm}
	
	\end{center}

\begin{abstract}
An abstract formulation of a duality 
principle established by Krasnoselskii is presented. Under appropriate conditions, it  shall be shown that,  if the solutions of a nonlinear functional equation can be obtained by finding fixed points of certain operators 
in possibly different Banach spaces, then these operators share some topological properties. 

\end{abstract}

\section{Introduction}

Let us consider the $T$-periodic problem for the system
\begin{equation}\label{1}
x'(t)=f(t,x(t))
\end{equation}
where $f:\R\times \R^n\to \R^n$ is smooth and $T$-periodic in its first coordinate. 
As is well known, there are several  approaches to search solutions as fixed points of certain mappings, specifically: 

 \begin{enumerate}
\item \textbf{In the phase space.} Define the Poincar\'e map $P:\textrm{dom}(P)\subset \R^n\to \R^n$ given by 
$P(x_0):=\Phi(T,x_0)$. Here, $\Phi$ denotes the associated flow, that is, $\Phi(t,x_0):=x(t)$, where  $x$ is the unique solution of 
(\ref{1}) satisfying the initial condition $x(0)=x_0$. Due to the periodicity of $f$, if $x_0$ is a fixed point of $P$ then the corresponding solution $x(t)$ is globally defined and $T$-periodic. Conversely, if $x(t)$ is a $T$-periodic solution of (\ref{1}), then $x_0:=x(0)\in \textrm{dom}(P)$ and $P(x_0)=x_0$. 

\item  \textbf{In a functional space}. 
Fix an appropriate Banach space $X$, e.g. 
$$X=C[0,T]:=C([0,T],\R^n)$$
or
 $$X=C_T:=\{x\in C(\R,\R^n): x(t+T)\equiv x(t)\}$$
 and define an operator $K:X\to X$ in such a way that the fixed points of $K$ can be identified with the set of $T$-periodic solutions of  (\ref{1}). A typical choice of such operators is $K_1:C[0,T]\to C[0,T]$ given by
 \begin{equation}\label{K1}
  K_1(x)(t):= x(T) + \int_0^t f(s,x(s))\,ds.
 \end{equation}
  \end{enumerate}
In both situations, solutions can be found 
by the use of different fixed points theorems or, more generally, employing degree theory. To this end, one may try, in the first case, to find an open bounded set $U$ such that $\overline U\subset \textrm{dom}(P)$ and $$\deg_B(I-P,U,0)\ne 0,$$
where $\deg_B$ denotes the Brouwer degree and, in the second case, 
an open bounded set $U_X\subset X$ such that
$$\deg_{LS}(I-K,U_X,0)\ne 0,$$
where $\deg_{LS}$ is now the (infinite dimensional) Leray-Schauder degree. 

At this point, it is worthy recalling the following paragraph of \cite{BK}:

\textit{Every approach to the specific problem in question has its advantages and drawbacks, and therefore the
picture we obtain when investigating a problem becomes more complete if we manage to find internal connections between different approaches and use the positive aspects of each of them. 
}

This is the goal of the so-called  theory of affinity, 
whose origin can be traced back to the 
works of M. A. Krasnoselskii  in the early sixties
(see
\cite{BK, K, krasno} and the references therein). 
In the previous 
context, a remarkable result is a relatedness principle, also called \textit{duality principle} (see \cite{maw}) which, 
roughly speaking, says that, despite their different nature,  
the   mapping $I-P$ and the operator $I-K$ have the same degree.  
In more rigorous terms, it is required that the sets $U$ and $U_X$ have \textit{common core}, which means: 

\begin{enumerate}
 \item $P$ and $K$ have no fixed points on $\partial U$ and $\partial U_X$ respectively. 
 \item If $x_0\in U$ is a fixed point of $P$, then the corresponding  solution $x(t)$ satisfies
 $\alpha(x)\in U_X$,  
 where  $\alpha$ maps the $T$-periodic functions  with their related elements of $X$. The definition of $\alpha$ depends on the choice of  $X$; for example, if $X=C_T$ then $\alpha(x):=x$ and if $X=C[0,T]$ then
 $\alpha(x):=x|_{[0,T]}$.
 \item If $x\in U_X$ is a fixed point of $K$, then $x_0:=x(0)\in U$. 
\end{enumerate}
 In other words, when $X=C[0,T]$, the last two properties say that if a $T$-periodic solution has initial value in $U$ then its restriction to $[0,T]$ belongs to $U_X$ and vice versa.  
 
In this context, one of the early results by Krasnoselskii reads
 $$\deg_B(I-P,U,0)=\deg_{LS}(I-K_1,U_1,0),
 $$
 where $K_1$ is the operator defined by (\ref{K1}) and the open bounded sets 
 $U\subset \R^n$ and $U_1\subset C[0,T]$ have common core. 
 This equivalence can be extended to other operators, both in $C[0,T]$ and $C_T$ or a variety of spaces, according to the equation under study.
 The proof relies on the homotopy invariance of the degree; however, this is not straightforward because $K_1$ and $P$ are defined in different spaces. In order to avoid this difficulty, the key idea consists in noticing that, if one identifies $\R^n$ with the subset of constant functions of $C[0,T]$, 
 then  $P$ can be identified with a compact operator $K_2$ which is, in turn, homotopic to $K_1$. 
 Then, the proof follows from the fact that, since $\textrm{Im}(K_2)\subset \R^n$, the Leray-Schauder degree of $I-K_2$ coincides with the Brouwer degree of $(I-K_2)|_{\R^n}$.
  Regarded from an abstract point of view,  this procedure 
can be understood as a conjugacy between operators and, as we shall see, this allows  to extend the relatedness properties to other situations.

The main goal of this paper is to give a general version of Krasnoselskii relatedness principle in abstract Banach 
spaces in such a way that, when applied to the specific 
cases treated in the literature, the original results are retrieved. We refer to \cite{K} for a very complete exposition, including various boundary conditions. 
The scope of the applications is not limited to ODEs: we shall present simple examples that include a nonlocal boundary problem for an elliptic PDE and a system of delay differential equations. 

The paper is organized as follows. In the next section, we shall present the abstract 
setting  of the Krasnoselskii results. 
Section 
\ref{lxnx} is devoted to the particular problem $L(x)=N(x)$ 
which includes a wide range of semilinear equations and is frequently treated in the context of the coincidence degree theory \cite{GM}. Finally, some applications are given in Section \ref{applic}. In particular, a relatedness principle for a delay equation shall not be deduced from the problem $L(x)=N(x)$ but, instead, from the abstract formulation given in section \ref{abstr-sec}. Finally, in section \ref{further} we introduce further comments
and pose some questions that shall be treated in subsequent works. 

Throughout the paper we shall use the generic notation `deg' to 
refer to the topological degree and reserve $\deg_{B}$ and 
$\deg_{LS}$ only for those cases in which we want to emphasize the fact that we are working in a finite or an infinite dimensional space, respectively. Also, 
since we shall be dealing with different spaces, in the second section we shall use, for the sake of clarity, the notation $\id_X$ to refer to the identity operator in the space $X$.

\section{The abstract result}
\label{abstr-sec}

In this section, we shall prove an abstract theorem that compares the degree 
of certain operators in a general context.  In the first place, let us establish  an useful lemma: 

 \begin{lemma}\label{conjug} (Conjugacy lemma)
Let $X$ and $\tilde{X}$ be Banach spaces, let $S$ be a subset of $X$ and let $U\subset X$ be  bounded and open such that  $S\cap \partial U=\emptyset$. Consider the diagram

 \begin{figure}[h]
    \begin{center}
\includegraphics[scale=.45]{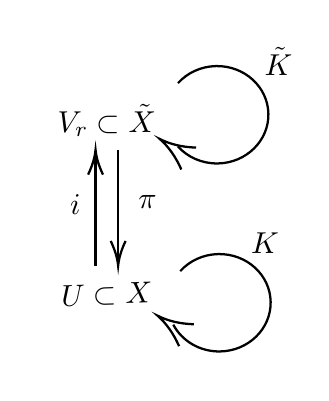}
\end{center}
\end{figure}

where:
\begin{itemize}
    \item $\pi$ and $ i$ are linear and continuous, with $\pi\circ  i=\id_{X}$,
    \item $K$ is compact such that $\fp(K)=S$, 
    where $\fp$ stands for the set of fixed points of a mapping,
    \item $\tilde{K}=i \circ K\circ \pi$
    \item $V_r=\pi^{-1}(U)\cap B(0,r)$.
\end{itemize}
Then $\fp(\tilde{K})= i(S)$. Moreover, there exists  $r_0>0$ 
such that if $r>r_0$ then  
$\partial V_r \cap  i(S)=\emptyset$ and 
$$
\dls(\id_X -K,U, 0)=\dls(\id_{\tilde{X}} - \tilde{K},V_r, 0).
$$
\end{lemma}
\begin{proof}
It is readily seen that $i\circ \pi:\tilde X\to \tilde X$ is a projector; thus, the subspace  $E:= i(X)\subset \tilde{X}$ is closed and $( i\circ \pi)\vert_E=\id_E$.
Let $\tilde{x}\in \fp(\tilde{K})$ and $x:=\pi(\tilde{x})$, then  
$$
\tilde{x}=( i\circ K \circ \pi )(\tilde{x}).
$$ 
Because $\tilde{x}\in E$, it follows that  $\tilde{x}= i(\pi(\tilde{x}))=  i(x)$ and hence  
$$
x=(\pi\circ i)(x)=\pi(\tilde x)= (K\circ \pi)(\tilde x)=K(x).
$$ 
Thus, $x\in S$  and, consequently, $\tilde{x}\in i(S)$.
Conversely, assume that $\tilde{x}= i(x)$ for some $x$ such that $x=K(x)$, then $\pi(\tilde{x})=x$ and 
$$\tilde{x}=i(x)= i(K(x))= i(K(\pi(\tilde{x})))=\tilde{K}(\tilde{x}).$$
{Since $U$ is bounded, we may fix $r_0$ such that $ i(U)\subset B(r_0,0)$ and  take $r>r_0$.}
It is clear that 
$\partial V_r\cap  i(S)=\emptyset$, 
$V_r \cap  i(S)= i(S\cap U)$ and $V_r\cap E= i(U)$.
Moreover, using the properties of the {Leray-Schauder degree}, we conclude:
\begin{align*}
    \dls(\id_{\tilde{X}}-\tilde{K}, V_r, 0)&=\dls((\id_{\tilde X}-\tilde{K})|_E,V_r \cap E, 0)\\
                                       &=\dls( i\circ (\id_X - K) \circ \pi, i(U), 0)\\
                                       &=\dls(\id_X-K,U, 0).
\end{align*}
\end{proof}
As a corollary, we obtain the following 
\begin{proposition}\label{conjugacion}
In the situation of the previous lemma, let $\tilde U\subset \tilde{X}$
be open and bounded such that $\partial \tilde U\cap  i(S)=\emptyset$ and $\tilde U\cap  i(S)= i(S\cap U)$. Then
$$
\dls(\id_{\tilde{X}} - \tilde{K},\tilde U,0)=\dls(\id_X - K,U,0).
$$
\end{proposition}
\begin{proof}
Let $V_r$ be as in  the previous lemma, then for $r\gg 0$ we know that  
$$
\dls(\id_{\tilde{X}} - \tilde{K},V_r, 0)=\dls(\id_X - K,U, 0).
$$
Because the fixed points of  $\tilde{K}$ in $\tilde U$ coincide with the fixed points of $K$ in $V_r$, the proof follows from the excision property of the degree. 
\end{proof}

In what follows, we shall consider
diagrams $\mathcal B$ of the form 
  
 \begin{figure}[h]
    \begin{center}
\includegraphics[scale=.45]{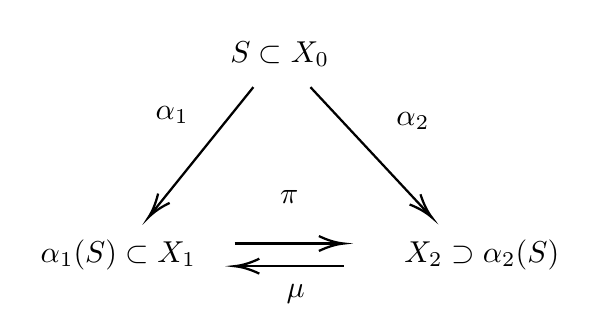}
\end{center}
\end{figure}
\noindent where:
\begin{itemize}
\item $X_0,X_1,X_2$ are Banach spaces.
\item $\alpha_1,\alpha_2$ are linear and continuous,  $\mu$ is compact,  $\pi$ is linear, continuous and onto. 
\item Restricted to $S$, $\alpha_1(S)$ and $\alpha_2(S)$, the diagram is commutative, namely:
$$\pi\circ \alpha_1(x)= \alpha_2(x),\qquad \mu \circ \alpha_2(x)= \alpha_1(x)
$$for every $x\in S$. 
\end{itemize} 
 
 In the applications,  $S$ shall be the set of  solutions of a  differential equation under 
certain boundary condition, 
 $X_0$  a space of functions satisfying the boundary condition and $X_1$, $X_2$ specific spaces in which solutions can be found as fixed points of accurate operators. For example, in the periodic case we may take $X_0=C_T$, 
 $X_1$ as the space of continuous functions over one period and $X_2$ as the phase space.

\begin{definition}
Let $U_i\subset X_i$  be open and bounded and consider the sets $S_i:=\alpha_i(S)\cap U_i$  with $i=1,2$. 
We shall say that $U_1$ and $U_2$ have \textbf{common core} with respect to the diagram $\mathcal{B}$ if
\begin{enumerate}
\item $\alpha_i(S)\cap \partial U_i =\emptyset$
\item There exists  $S_0\subset S$ such that $\alpha_i (S_0)=S_i$.
\end{enumerate}
\end{definition}

{\begin{obs}
In particular, the last condition ensures that 
$$(\mu\circ\pi)|_{S_1}= \id_{S_1},\qquad (\pi\circ\mu)|_{S_2}= \id_{S_2}.
$$
\end{obs}}
\noindent In the previous setting, we may also consider fixed point operators 
$K_i:X_i\to X_i$, that is, such that 
 $\fp(K_i)=\alpha_i(S)$. 
 The resulting diagram $\mathcal B$ shall be called a 
 \textbf{Krasnoselskii diagram}
if: 
\begin{enumerate}
    \item $U_i$ have common core with respect to $\mathcal{B}$.
    \item $K_i$ are fixed point operators.
    \item $\pi$  has a continuous right inverse  $i\in L(X_2,X_1)$, that is:  $\pi\circ i=\id_{X_2}$.
 
\end{enumerate}
 
 \begin{figure}[h]
    \begin{center}
\includegraphics[scale=.45]{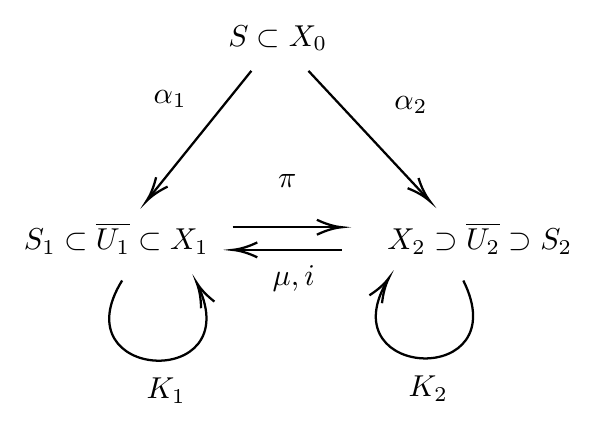}
\end{center}
\end{figure}

{As a special case, taking  
$$K_1:=\mu\circ\pi, \qquad K_2:=\pi\circ\mu
$$
it is readily verified that 
$\alpha_i(S)\subset \fp(K_i)$. Indeed, 
if $x=\alpha_1(\tilde x)$, then 
$$(\mu\circ\pi)(x) = (\mu\circ\pi\circ\alpha_1)(\tilde x) = (\mu\circ\alpha_2)(\tilde x)= \alpha_1(\tilde x)=x.  
$$
The inclusion $\alpha_2(S)\subset \fp(K_2)$ is analogous.}
 
 This allows to give a direct proof of the following result, which can be regarded as a particular case of \cite[Thm 26.1]{krasno} for the
composition of arbitrary operators in Banach spaces.  
 
\begin{theorem}\label{teorema principal} 
In the previous situation, assume that $\mathcal B$ is a   
Krasnoselskii diagram for  $K_1=\mu\circ\pi$ and $K_2=\pi\circ\mu$. Then $$\dls(\id_{X_1}-K_1,U_1,0)=\dls(\id_{X_2}-K_2,U_2,0)$$
\end{theorem}

\begin{proof}
Define $\tilde{K}:=i\circ K_2 \circ \pi$, then, by Lemma \ref{conjug},  there exists $r_0$ such that, for  $r>r_0$ and $V_r=\pi^{-1}(U_2)\cap B(0,r)$, 
  $$ \dls(\id_{X_1}-\tilde{K},{V}_r,0)=
  \dls(\id_{X_2}-K_2,{U_2},0).$$ 
  Fix $r>r_0$ such that $i(U_2)$ and $U_1$ are contained in $B(0,r)$ and consider the homotopy $H_\lambda:=\lambda K_1 + (1-\lambda) \tilde{K}$.
  
  \textbf{Claim}: The homotopy is admissible, that is: $\fp(H_\lambda)\cap \partial V_r=\emptyset$ for all $\lambda \in [0,1]$. 
  Indeed, let   $\lambda \in [0,1]$ and $x\in\overline{V_r}$ such that
\begin{equation}\label{hom}
 x  =\lambda \   K_1 ( x)+ (1-\lambda) \   \tilde{K} ( x).\end{equation} 
 Because $\pi({V}_r)={U_2}$, it follows that $\pi (x)\in \overline{U_2}$.  
 Next, apply $\pi$ on both sides of  ($\ref{hom}$) to obtain 

\[ 
\begin{array}{ll}
 \pi  (x)  &= \lambda \ (\pi \circ K_1) (x)+ (1-\lambda) \ (\pi \circ \tilde{K}) (x) \\
             &=\lambda \ (\pi \circ  \mu \circ \pi) (x)+ (1-\lambda) \ (\pi \circ  i \circ \pi \circ \mu \circ \pi) (x) \\
             &=\lambda \ K_2  (\pi (x))+ (1-\lambda) \  K_2  (\pi (x)) \\
             &= K_2 (\pi (x)).\ \
\end{array}
\]
Thus,  $\pi (x) \in \fp{}(K_2|_{\overline{U_2}})=S_2\subset U_2$, whence $x \in \pi^{-1}(U_2)$. Moreover, 
it is seen from 
   (\ref{hom}) that
$$\Vert x\Vert \leq \lambda \Vert  K_1 (x) \Vert+ (1-\lambda) \Vert  \tilde{K} (x) \Vert. $$
 Since $U_1$ and $U_2$ have common core, 
 $\pi(x)=S_2=\alpha_2(S_0)$, then
 $$K_1(x)=\mu\circ \pi (x)\in \alpha_1(S_0)=
 S_1\subset U_1\subset B(0,r).$$
 On the other hand, because $\pi(x)\in \fp(K_2)$ we obtain:
$$ \tilde{K}(x)=(i \circ K_2 \circ\pi)(x)=i(\pi(x)) \in i(U_2)\subset B(0,r).$$
 We conclude that $\|x\|<r$ and the claim follows.
 
Thus, using the homotopy invariance of the degree, we deduce:
$$\dls(\id_{X_1}-\tilde{K},V_r,0)=\dls(\id_{X_1}-K_1,V_r,0).$$
Finally, from the fact that $V_r\cap \alpha_1(S)=U_1 \cap \alpha_1(S)$ and the excision property it follows that
$$\dls(\id_{X_1}-K_1,V_r,0)=\dls(\id_{X_1}-K_1,U_1,0)
$$.\end{proof}

\section{The problem $L(x)=N(x)$}
\label{lxnx}

Many semilinear boundary value problems for differential equations can be expressed 
in the form
\begin{align}
L(x)&=N(x)\label{Eq dif pvi 1}
\end{align}
as described by the following  diagram:

 \begin{figure}[ht]
    \begin{center}
\includegraphics[scale=.45]{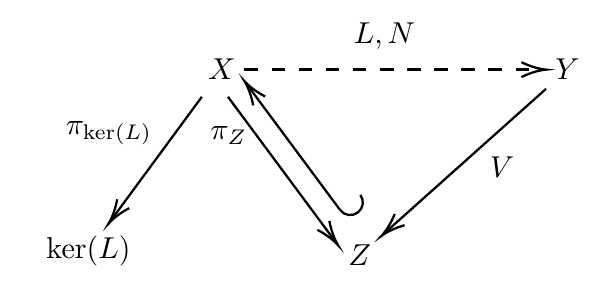}
\end{center}
\end{figure}

\newpage

Here, 
\begin{itemize}

\item $L:D\to Y$ is a linear  operator, where 
$D:=\textrm{dom}(L)\subset X$ is  dense and $L$ is onto the Banach space $Y$.
\item $\ker(L)$ is closed and there exists a continuous projector  
$\pi_{\ker (L)}:X\to X$ onto $\ker(L)$. It follows that   
$X=\ker(L) \oplus Z$, where $Z=\ker(\pi_{\ker(L)})$ is closed.


\item $V:Y\to Z$ is linear and compact such that  
 {$L\circ V=\id_{Y}$ and $\textrm{Im}(V)= Z$}.
\item $N$ is continuous.
\end{itemize}

{
\begin{obs}  {In particular, due to the Hahn-Banach Theorem, a continuous projector $\pi_{\ker(L)}$ always exists when $\ker(L)$ is finite dimensional.} In some situations, it may occur that $V$ is not compact, although
all the results are still valid under the weaker assumption that $V\circ N$ is compact. This is analogous to what is called, in the context of the coincidence degree theory, the $L$-compactness of the operator $N$. 
\end{obs}
}

\begin{proposition}
In the previous setting, let $ K_0:=\pik + V\circ N$.  Then the solutions of  (\ref{Eq dif pvi 1}) are the fixed points of $K_0$. In other words,  (\ref{Eq dif pvi 1}) is equivalent to the equation \begin{equation}\label{eq int}
K_0(x)=x.
\end{equation}\end{proposition}
\begin{proof}
Let $x\in \fp(K_0)$, then $x=\pik(x)+(V\circ N)(x)$ and, applying $L$, we deduce that $L(x)=N(x)$. Conversely, 
let $x$ be a solution of (\ref{Eq dif pvi 1}) and set  $z:= x-\pik(x)\in Z$. It follows that $L(z)=L(x)=N(x)$; thus,  $z=(V\circ N) (x)$, that is,  $x=\pik(x)+(V\circ N)(x)$.
\end{proof}

Finally, we shall need a condition regarding the uniqueness of the solutions. This is the role of the following assumption, that can be interpreted as an  
abstract  Gronwall-like inequality. Specifically, we shall say that the previous  diagram satisfies 
the condition $(J)$ if, for all $\eta\ge 0$, 
\begin{equation}\label{(J)} \tag{J}
x-y=\eta V(N(x)-N(y))\quad \Longrightarrow\quad x=y.
\end{equation}



\subsection*{Special solutions}

In this section, we shall analyze the problem of finding solutions of  (\ref{Eq dif pvi 1}) that satisfy a certain linear condition, e. g. periodic solutions of an ODE or a DDE, Dirichlet or Neumann boundary conditions for ODEs or PDEs, etc.  

Let us consider the problem
\begin{align}
L(x) &= N(x) \label{sistema especial 1}\\ 
\delta (x) &=0\label{sistema especial 2}
\end{align}
where $\delta: X \to X$ is a bounded linear operator such that $\textrm{Im}(\delta)\subset \ker(L)$. 
As an example, we may consider the Dirichlet problem 
\begin{align}
x''(t) = & f(x(t),x'(t)) \label{Dirichlet homogeneo 2 orden 1}\\
x(0)= & \, x(1)=0
\label{Dirichlet homogeneo 2 orden 2},
\end{align}
with $X=C^1[0,1]$. Observe that $\ker(L)=\lbrace b+ta:  a, b \in \R^n \rbrace\cong \R^{2n}$. 
Thus, if we define
$\pik (x):=x(0)+tx'(0)\cong (x(0),x'(0))$ and  $\delta(x)=x(0)+tx(1)$, 
then 
(\ref{Dirichlet homogeneo 2 orden 1})-(\ref{Dirichlet homogeneo 2 orden 2}) is written in the form (\ref{sistema especial 1})-(\ref{sistema especial 2}) .  


\begin{proposition}\label{proposicion 2}
In the previous situation, 
define $\pi:=\pik+\delta$ and $K:=\pi + V\circ N$, then  (\ref{sistema especial 1})-(\ref{sistema especial 2}) is equivalent to the fixed point problem $K(x)=x$.
\end{proposition}
\begin{proof}
Let $x\in\fp(K)$, then  $x=\pi(x)+(V\circ N)(x)$. Because $\textrm{Im}(\pi)\subset\ker(L)$, applying $L$ we obtain
$L(x)=N(x)$. On the other hand, if we apply   $\pik$ instead, we deduce that   $\pik(x)=\pik(x)+\delta(x)$ which, in turn, implies $\delta(x)=0$.
Conversely, if $x$ solves (\ref{sistema especial 1}) -(\ref{sistema especial 2}), then  $x=\pik (x)+(V\circ N)(x)$. Since $\delta(x)=0$, 
this obviously implies 
$x=\pik(x)+\delta(x)+(V\circ N) (x)=K(x)$.
\end{proof} 

\begin{lemma}\label{lemma diagrama base}
Consider as before a diagram 
 associated to (\ref{sistema especial 1})-(\ref{sistema especial 2}) 
 {satisfying (J)}
 and 
 assume there  exists $\mu: \ker(L)\to X$ continuous such that
\begin{enumerate}
\item $\mu(\ker(L))=\{x\in X: x\, \hbox{is a solution of (\ref{sistema especial 1})}\}$.
\item $\pik\circ\mu=\id_{\ker(L)}$.
\end{enumerate}
Let $\mathcal{B}$ be the  diagram given by 
 \begin{figure}[h]
    \begin{center}
\includegraphics[scale=.45]{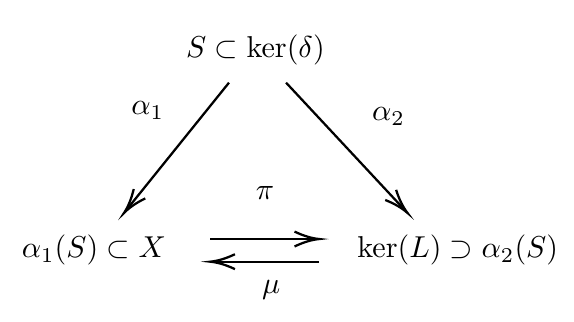}
\end{center}
\end{figure}

\noindent where: \begin{itemize}
    \item $S=\lbrace x\in X: x \text{ is a solution of } \hbox{(\ref{sistema especial 1})-(\ref{sistema especial 2})}\rbrace$.
    \item $\pi=\delta +\pik$
    \item $\alpha_1$ denotes the inclusion of $\ker(\delta)$ in $X$ 
    \item $\alpha_2=\pik|_{\ker(\delta)}$.
\end{itemize}
Then $\mathcal{B}$ 
 commutes when restricted to 
 $S$, $\alpha_1(S)$ and $\alpha_2(S)$.
 
\end{lemma}
\begin{proof} 
Let $x\in S$, then    
$$(\pi \circ \alpha_1)(x)=\pi(x)=\delta(x)+\pik(x)=\pik(x)=\alpha_2(x).$$ On the other hand, calling $k:=\pik(x)$,  it is seen that $(\mu\circ \alpha_2)(x)=\mu(k)$. It follows from the definition 
  $L(\mu(k))=N(\mu(k))$ and  
$\pik(\mu(k))=k=\pik(x)$. 
{Because $x\in S$, it follows that 
$x-\mu(k)= V(N(x)- N(\mu(k)))$ and we deduce from $(J)$ that $x=\mu(k)=(\mu\circ \alpha_2)(x)$. }
\end{proof}

\begin{lemma}\label{diagrama de Krasnoselski}
In the situation of the previous Lemma, 
{assume that the mapping    $\pi=\delta+\pik$ is onto.} 
Furthermore, assume that $U_1\subset X$ and $U_2\subset\ker(L)$ are bounded 
open sets with common core 
with respect to $\mathcal{B}$.  
Consider the diagram $\mathcal{K}$ given by
  \begin{figure}[h]
    \begin{center}
\includegraphics[scale=.45]{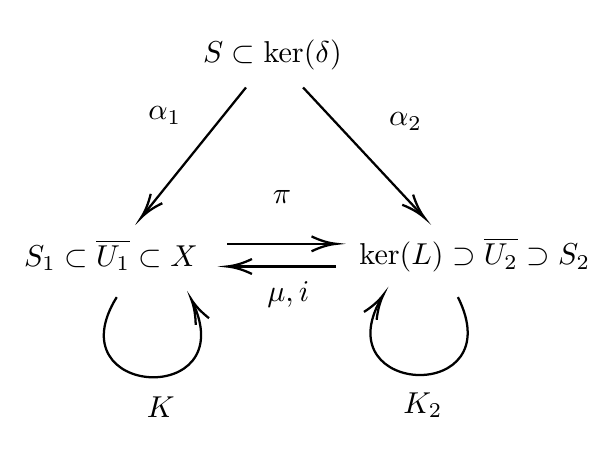}
\end{center}
\end{figure}

with $S$, $\pi$, $\alpha_1$ and $\alpha_2$ as before
and
\begin{itemize}
\item $i$ is a right inverse of $\pi$
\item $K=\pi + V\circ N$. 
 \item $K_2=\pi\circ \mu$.
\end{itemize}
Then $\mathcal K$ is  a Krasnoselskii diagram. 
\end{lemma}
\begin{proof}
From the previous lemma, we only have to verify that $K$ and $K_2$ are fixed point operators. 
This is clear for $K$, from Proposition \ref{proposicion 2} and the fact that  $\alpha_1$ 
is simply the inclusion. 
{Since we already know that $ \alpha_2(S)\subset \fp(K_2)$}, 
it only remains to prove that  $\fp(K_2)\subset \alpha_2(S)$.
To this end, let us  consider  
$k\in \fp(K_2)$, that is, $k=(\pi\circ \mu) (k)$.
Let $x=\mu(k)$, then $x$ is a solution of  (\ref{sistema especial 1}) and satisfies 
  $\alpha_2(x)=\pik(x)=k$. Moreover, 
  $$k=\pi(\mu(k))=
  (\pik+\delta)(\mu(k))=k+\delta(\mu(k))=k+\delta(x).$$ We deduce that $\delta(x)=0$, whence 
  $x\in S$ and $k=\alpha_2(x)\in \alpha_2(S)$.
\end{proof}

\begin{lemma}\label{kras-diag-1}
In the situation of the previous lemma, consider the diagram $\mathcal K_1$ obtained from $\mathcal K$ by replacing the operator $K$
by $K_1:=\mu\circ\pi$. 
Then $\mathcal K_1$ is  a Krasnoselskii diagram. 
\end{lemma}

\begin{proof}
Let $x\in \fp(K_1)$, that is $x=(\mu \circ \pi)(x)=\mu(\pik(x)+\delta(x))$. Applying $\pik$ at both sides, we obtain  $\pik(x)=\pik(x)+\delta(x)$ and consequently  $\delta(x)=0$. 
Moreover, since $x=\mu(\pi(x))$, 
we deduce that $x$ is a solution of (\ref{sistema especial 1}) and hence $x\in S=\alpha_1(S)$. {Thus, the proof is complete, because we already know that $\alpha_1(S)\subset \fp(K_1)$}. 
\end{proof}

We are now in condition of establishing the main result of this section.

\begin{theorem}\label{teorema Lx=Nx}
In the situation of the previous lemmas, assume that $\mu$ is compact.  
Then
$$\dls(I-K,U_1,0)=\dls(I-K_2,U_2,0).$$
\end{theorem}
 \begin{proof}
From the previous lemma and Theorem \ref{teorema principal}, we know that $$\dls(I-K_1,U_1,0)=\dls(I-K_2,U_2,0). $$
Thus, it suffices to verify that $$\dls(I-K_1,U_1,0)=\dls(I-K,U_1,0).$$
To this end, let us define $H_\lambda:= \lambda K + (1-\lambda)K_1$ over the set $\overline{U_1}$ and suppose that   $x \in \fp(H_\lambda)$, that is
\begin{equation}
x=\lambda \left[\pi(x)+V\circ N (x)\right] + (1-\lambda)\left( \mu \circ \pi\right) (x).
\end{equation}
On the one hand, applying $\pik$ we obtain:
\begin{equation}
\pik(x)=\lambda \pi(x)  + (1-\lambda) \pi (x) =\pi(x)\implies\delta(x)=0.
\end{equation}
On the other hand, applying $L$ instead, we obtain:
\begin{equation}\label{cucu}
L(x)=\lambda   N (x)  + (1-\lambda)  (L \circ \mu \circ \pi) (x) 
\end{equation}
Let $y=(\mu \circ \pi)(x)$, then it follows from the definition of $\mu$ that  $L(y)=N(y)$ and $\pik(y)=\pi(x)$; moreover, because $\delta(x)=0$, it follows that 
$\pik(y)=\pik(x)$. 
Thus, writing (\ref{cucu}) as $L(x)=\lambda N(x) + (1-\lambda)L(y)$, we deduce
the equality 
\begin{equation}
L(x)-L(y)=\lambda \left( N(x)-N(y)\right),
\end{equation}
that is
\begin{equation}
x-y=\lambda V(N(x)-N(y))
\end{equation}
and, from $(J)$, we conclude that 
$x=y$ and $L(x)=N(x)$. 
This implies that  $x\in S=\alpha_1(S)$ and 
hence  $x\notin\partial U_1$. Thus, the proof follows from the homotopy invariance of the Leray-Schauder degree. 
\end{proof}





\section{Applications}\label{applic}
\subsection{Periodic solutions of ODEs}
\label{period}
Consider the problem
\begin{align}
x'(t)&=f(t,x(t))\label{EDO}\\
x(0)&=x(T)\label{CP}
\end{align} 
where $f:\mathbb R\times\R^n\to\R^n$ is continuous  and $T$-periodic in $t$. For simplicity, we shall also assume that $f$ is globally  Lipschitz in $x$ with constant $l$. 
A  diagram 
for this equation is the following
 \begin{figure}[h]
    \begin{center}
\includegraphics[scale=.45]{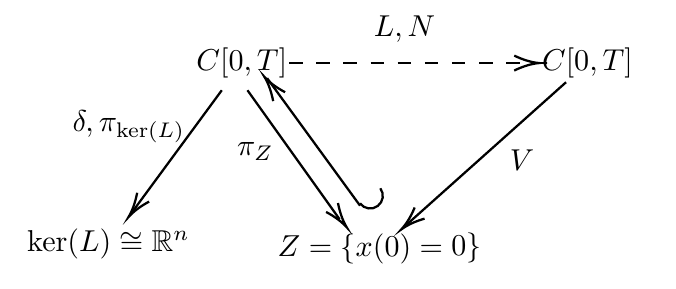}
\end{center}
\end{figure}

\noindent where:
\begin{itemize}
\item $L:\textrm{dom}(L)\subset  C[0,T] \to C[0,T]$ is given by
$Lx=x'$ and 
{$\textrm{dom}(L)=C^1[0,T]$.}

\item $N:C[0,T]\to C[0,T]$  is the Nemitskii operator associated to $f$, namely  $N(x)(t):=f(t,x(t))$.
\item $\ker(L)=\lbrace x(t)\equiv c: c\in \R^n\rbrace\cong \R^n$.
\item $Z=\lbrace x\in C[0,T]\mid x(0)=0\rbrace$. 
\item $V(x)(t):=\int_0^t x(s)ds$.
\item $\pik(x)\equiv x(0)$.
\item $\delta(x):=x(T)-x(0)$
\end{itemize}
Here, we may consider the flow $\Phi$ associated to the system, that is $\Phi(t,x_0)=x(t)$, where 
$x$ is the unique solution of (\ref{EDO}) such that $x(0)=x_0$, and define 
$\mu(x_0)(t):=\Phi(t,x_0)$. 
Since $\ker(L)\cong \R^n$, we may assume   $\mu:\ker(L)\to C[0,T]$. From the
Arzel\`a-Ascoli Theorem, it is clear that $\mu$ is compact.  
\begin{lemma}
Condition $(J)$ is satisfied.
\end{lemma}
\begin{proof}
Let $x,y\in C[0,T]$ satisfy  $(x-y)(t)=\eta\bigintssss_0^t \left( f(s,x(s)) - f(s,x(s))\right)ds$ for  some $\eta\geq0$, then 
\begin{align}
\vert x(t)-y(t)\vert &=\eta\left\vert\int_0^t \left( f(s,x(s)) - f(s,y(s))\right)ds\right\vert \\
								&\leq \eta \int_0^t \left\vert f(s,x(s)) - f(s,y(s))\right \vert ds\\
                                &\leq   \eta l \int_0^t \vert x(s) - y(s)\vert ds.
\end{align}
From  Gronwall's lemma, we deduce that $\vert x(t)-y(t)\vert=0$ for all $t$. 
\end{proof}  
\begin{lemma}
$\pi=\pik+\delta$ is onto.
\end{lemma}
\begin{proof}
This is trivial from the definition, because  
$\pi(x)=x(T)$.
\end{proof}

Thus, we may consider a diagram $\mathcal B$ as in Lemmas \ref{lemma diagrama base} and
\ref{diagrama de Krasnoselski} 
and assume that     $U_1\subset C[0,T]$ and $U_2\subset \ker(L)$ have common core with respect to $\mathcal{B}$. Here, it is seen that 
$$K(x)(t)= \pi(x)+ (V\circ N)(x)(t)=x(T)+\int_0^t f(r,x(r))dr$$ and $$K_2(x_0)=(\pi\circ\mu)(x_0)=\Phi(T,x_0)$$
coincides, via the isomorphism $\ker(L)\cong \R^n$,
with the Poincar\'e map. 
Then, by Theorem 
\ref{teorema Lx=Nx}, 
 the original result by  Krasnoselskii mentioned in the introduction is retrieved: 
$$\dls_{LS}(I-K,U_1,0)=\dls_B(I-P,U_2,0).$$
Moreover, observe that 
$$K_1(x)(t)=(\mu\circ\pi)(x)(t)= \Phi(t,x(T)),
$$
which yields another one of the equivalences presented in \cite{K}:
$$\dls_{LS}(I-K_1,U_1,0)=\db_B(I-P,U_2,0).$$

\subsubsection*{Other fixed point operators}

In the previous example, it is observed that the operator   $\pi(x)=x(T)$ is a projector onto  $\ker(L)$. With this in mind, we may extend the definition of $K$ to  a more general family of fixed point operators: 
\begin{lemma}
In the previous situation, let $\gamma:C[0,T]\to C[0,T]$ be a projector onto $\ker(L)$ and define the (compact) operator
 $$ K_\gamma :=\gamma + [(\pi - \gamma) + \id_X]\circ V\circ N $$
  Then $K_\gamma$ is a fixed point operator respect to
$\mathcal{B}$ and $$\dls_{LS}(I-K_\gamma,U_1,0)=\db_B(I-P,U_2,0).$$
\end{lemma}

\begin{proof}
Let $x\in \fp(K_\gamma)$, then
\begin{equation}\label{ident-gamma}
    x=\gamma(x)+ [(\pi - \gamma)\circ V\circ N] (x)+(V\circ N) (x)
\end{equation}
and applying $L$ we obtain  $L(x)=N(x)$. On the other hand, applying $\gamma$ we obtain
$$\gamma(x)=\gamma(x)+((\pi - \gamma)\circ V\circ N) (x)+(\gamma\circ V\circ N) (x)$$
and hence 
$$
(\pi\circ V\circ N)(x)=0.
$$
Finally, applying 
  $\pik$, from the previous equality and employing the fact that $\pik\circ V\equiv 0$ we obtain 
  $$   \pik(x)=\gamma(x)-(\gamma \circ V\circ N) (x).
$$
 In the same way, because  $\pi$ is a  projector 
 we also obtain from (\ref{ident-gamma}) that
  $$\pi(x)=\gamma(x)-(\gamma \circ V\circ N) (x)$$
and consequently  $\delta(x)=\pi(x)-\pik(x)=0$, that is,  $x\in\alpha_1(S)$. 

\noindent Conversely, let $x\in \alpha_1(S)$ and write 
 $C[0,T]= \ker(\gamma)\oplus\textrm{Im}(\gamma)=
 \ker(\gamma)\oplus\ker(L)$. The operator $V_\gamma:=V-\gamma \circ V$ is a right  inverse of $L$ with range in $Z_\gamma:=\ker(\gamma)$, then the equation $L(x)=N(x)$ is equivalent to  $x=\gamma(x) + (V_\gamma\circ N)  (x)$, that is 
\begin{equation}\label{sss}
x=\gamma(x)+(V\circ N) (x)-(\gamma\circ V\circ N) (x).    
\end{equation}
Moreover, we know from Proposition \ref{proposicion 2} that
$$
x=\pi(x)+(V\circ N) (x)
$$ and applying  $\pi$ we deduce that $(\pi\circ V\circ N) (x)=0$. 
Combined with (\ref{sss}), this implies   
$$
x=\gamma(x) + [(\pi-\gamma)\circ V\circ N] (x)+(V\circ N) (x)    
$$ and hence $x\in \fp(K_\gamma)$.

Next, consider the homotopy $H_\lambda=\lambda K +(1-\lambda)K_{\gamma}$ with $\lambda \in  [0,1]$. It is clear that $H_\lambda=K_{\tilde{\gamma}}$ for $\tilde{\gamma}=\lambda \pi +(1-\lambda)\gamma$, which is also a projector. Thus, $\fp(H_\lambda)=\alpha_1(S)$, which does not intersect  $\partial U_1$; hence the degree is well defined along the homotopy and 
$$\dls_{LS}(I-K_\gamma,U_1,0)=\dls_{LS}(I-K,U_1,0)=\db_B(I-P,U_2,0).$$
\end{proof}

 A different operator can be obtained by simply noticing 
 that, if $L(x)=N(x)$, then $x$ satisfies the boundary condition if and only if $\overline {N(x)}=0$, where $\overline \varphi$ denotes the average of a continuous function $\varphi$, namely $\overline \varphi:=\frac 1T\int_0^T\varphi(t)dt$. This motivates the  definition of
$$K_3(x):=\overline{x}+T \overline{N(x)}+V[ N(x)-\overline{N(x)}]-\overline{V [N(x)-\overline{N(x)}]},
$$
which is clearly compact and, as one easily 
verifies,  it is a fixed point operator. 

\begin{corollary}\label{k3} In the previous situation, 
$$\dls_{LS}(I-K_3,U_1,0)=\db_B(I-P,U_2,0).$$
\end{corollary}
\begin{proof}
Let  $K_4(x):=\overline{x}+T \overline{N(x)}+V(N(x))-\overline{V( N(x))}$, then  $K_4=K_\gamma$ for the projector $\gamma(x):=\overline{x}$. Then $K_4$ is a fixed point operator and 
$$\dls_{LS}(I-K_4,U_1,0)=\db_B(I-P,U_2,0).$$
Next, consider the homotopy 
$$
H_{\lambda}(x)=\overline{x}+T \overline{N(x)}+V\left( N(x)-\lambda \overline{N(x)}\right)-\overline{V \left(N(x)-\lambda \overline{N(x)}\right)},
$$
which is also a fixed point operator for all $\lambda\in [0,1]$. As before, we deduce that $H_\lambda$ is admissible and $H_0 =K_4$, whence $$\dls_{LS}(I-K_3,U_1,0)=\dls_{LS}(I-K_4,U_1,0)=\db_B(I-P,U_2,0).$$
\end{proof}
{The latter operator is closely related to the  standard Lyapunov-Schmidt decomposition for the problem $x'(t)= f(t,x(t))$ in $C_T$, which can be 
expressed as
$$\overline{N_Tx}=0,\qquad  x -\overline x =\tilde V({N_Tx}-\overline{N_Tx} )
$$
where $N_T:C_T\to C_T$ is given as before by $N_T(x)(t):=f(t,x(t))$
and the linear operator 
$\tilde V$ is defined for $\varphi\in C_T$ such that $\overline\varphi=0$ as  the unique $x\in C_T^1$ such that $x'=\varphi$ and $\overline x=0$. 
Thus, the problem is equivalently written as
$$x= \overline x + T\overline {N_T(x)} + \tilde V({N_Tx}- \overline{N_Tx}):= K_5(x).
$$
In order to prove the relatedness principle for this operator, it proves convenient to notice that 
 $$K_5(x)=K_3\left(x|_{[0,T]}\right)$$
 extended $T$-periodically to $\mathbb R$. 
 Moreover, we may define 
 $$\tilde S:=\{ x\in C_T: \hbox{$x$ is a solution of (\ref{sistema especial 1})}\}=\fp(K_5)$$
 and recall that $S$ is the set of solutions of (\ref{sistema especial 1})-(\ref{sistema especial 2}) in $X=C[0,T]$. 
 }

\begin{proposition}
Let $\tilde U\subset C_T$ and $U\subset C[0,T]$ be open bounded sets such that
\begin{itemize}
    \item $\partial \tilde U\cap \tilde S=\emptyset$ and $\partial U\cap S=\emptyset$.   
    \item For $x\in C_T$, 
    $$x\in \tilde U\cap \tilde S \iff x|_{[0,T]}\in U\cap S.
    $$
 \end{itemize}
Then 
$$
\dls_{LS}(I-K_3,U,0)=\dls_{LS}(I-K_5,\tilde U,0).
$$
\end{proposition}
\begin{proof}
{Let $X_0=\ker(\delta)=\{x\in C[0,T]:x(0)=x(T)\}$, then $\textrm{Im}(K_3)\subset X_0$. From the properties of the degree, we know that 
$$\dls_{LS}(I-K_3,U,0)=\dls_{LS} ((I-K_3)|_{X_0},U\cap X_0,0).
$$
Next, set $\tilde\pi:C_T\to X_0$ as the restriction and $\tilde i:X_0\to C_T$ as the $T$-periodic extension, then the result is a direct consequence of Proposition \ref{conjugacion} and the fact that $K_5= \tilde i\circ K_3\circ \tilde \pi$.}
\end{proof}

{Let us consider now the fixed point operator defined as follows. For arbitrary $\eta\ne  0$, let $K^\eta:C_T\to C_T$ be given by  $K^\eta(x)= y$, where $x$ is the unique $T$-periodic solution of the problem  
$$y'(t) + \eta y(t) = f(t,x(t)) + \eta x(t).$$
In order to analyze the affinity with respect to the preceding operators, let us consider firstly the case  $\eta>0$ and take again the linear homotopy 
$$H_\lambda(x):= \lambda K^\eta(x) + (1-\lambda) K_5(x).  $$
Suppose  $x=H_\lambda(x)$, then taking average we obtain 
$$\overline x = 
\lambda \left[\overline x + \frac {\overline{N_Tx}}\eta \right]
+(1-\lambda)[\overline x + T\overline{N_Tx}]
$$
that is
$$0=(1-\lambda)  T\overline{N_T(x)}  + \lambda\frac {\overline{N_T(x)}}\eta,
$$
whence 
$\overline{N_Tx}=0$. Setting $y=K^\eta(x)$, it follows that
$$x' = N_T(x) + \lambda  \eta(x-y)$$
and
$$y' = N_T(x) + \eta (x-y).
$$
Thus, $$
(x-y)'  = (\lambda-1)(x-y).
$$
When $\lambda <1$, 
 the fact  $x-y\in C_T$ implies $x=y$ and hence $x'=N_T(x)$; on the other hand, the case $\lambda=1$ corresponds to a fixed point of $K^\eta$, so we also deduce  that $x'=N_T(x).$ In consequence, $x$ is a solution and   $K^\eta$ is homotopic to $K_5$, provided that there are no solutions on $\partial \tilde U$. } 

{Now assume, instead, that $\eta<0$, then we may define 
$$\hat K_5(x):= \overline x- T\overline{N_T(x)} + \tilde K({N_T(x)-\overline{N_Tx}}) $$
and it follows exactly as before that $K^\eta$ is homotopic to $\hat K_5$. Furthermore, if we 
define $\hat K_3$ by changing in $K_3$ the sign of the term   $T\overline N_T(x)$ and 
set $U$ and $\tilde U$  as before, then
  $$\deg_{LS}(I-\hat K_5,\tilde U,0)=\deg_{LS}(I-\hat K_3,   U,0).
$$

 {Next we may take,  in the original diagram, $\delta(x):=x(0)-x(T)$. Then the mapping  $$\pi(x)=2x(0)-x(T)$$ 
is, as before, a projector and taking $\gamma(x)=\overline x$
we obtain the corresponding operators
$$K_\gamma(x)=\overline x + [-V(N(x))(T) - \overline{V(N(x))} + V(N(x))]=\hat K_3(x)$$
and
$$K_2(x_0)=2x_0 - \Phi(T,x_0).$$
Thus, 
$$\deg_{LS}(I-\hat K_3, U,0) = \deg_B(I-K_2,U_2,0).
$$
Moreover, observe that 
$$I-K_2(x_0)= \Phi(T,x_0)-x_0=-(I-P)(x_0),$$}
whence
$$
\deg_B(I-K_2,U_2,0)= (-1)^n\deg_B(I-P,U_2,0).
$$

Summarizing, we have established: 
\begin{proposition}\label{eta-period}
Let $U$ and $\tilde U$ be as before, then
$$\deg_{LS}(I-K^\eta,\tilde U,0)= 
sgn(\eta)^n \deg_{LS}(I-K_3, U,0).
$$
\end{proposition}
}

\subsection{Second order equation with Dirichlet conditions}
\label{dirich}

Let us consider the problem 
\begin{align}
x'' = & f(t,x(t)) \label{Dirichlet homogeneo 2 orden 1-bis}\\
x(0)= \ & x(1)=0
\label{Dirichlet homogeneo 2 orden 2-bis},
\end{align}
where  $f:[0,T]\times \R^n\to \R^n$ is continuous and Lipschitz in $x$ with constant $l$. 
A diagram  for this equation reads
 \begin{figure}[h]
    \begin{center}
\includegraphics[scale=.45]{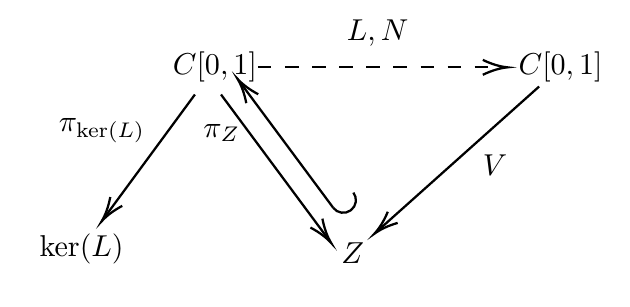}
\end{center}
\end{figure}
\newpage 
\noindent where:
\begin{itemize}
\item  $X=C^1[0,1]$ and  $Lx:=x''$ is defined in $\textrm{dom}(L)=C^2[0,T]$. 
 \item $N:X\to X$ is the Nemitskii operator $N(x)(t):=f(t,x(t))$.
\item $\ker(L)=\lbrace x(t)=ta+b: a\in \R^n,b\in \R^n\rbrace\cong \R^{n}\times \R^{n}$.
\item $Z=\lbrace x\in C^1[0,1]: x(0)=x'(0)=0\rbrace$. 
\item $Vx(t)=\int_0^t \int_0^r x(s)dsdr$.
\item $\pik(x)=tx'(0)+x(0)\cong (x'(0),x(0))$. 
\item $\delta(x)=tx(1)+x(0) \cong (x(1),x(0))$. 
\end{itemize}

In this setting, we may define $\mu(ta+b)$ as the unique solution of (\ref{Dirichlet homogeneo 2 orden 1-bis}) with initial condition $x(0)=b$,  $x'(0)=a$. It is clear that $\mu$ is well defined and compact. 

The following lemma is a direct consequence of  Gronwall's inequality:
\begin{lemma}\label{Gronwall de segundo orden}
Assume that $u\in C[0,1]$ satisfies  
\begin{equation}\label{hipotesis Gronwall segundo orden}
0\le u(t)\leq c^2 V(u)(t) \qquad \forall t \in [0,1].
\end{equation}
Then $u\equiv 0$.
\end{lemma}

\begin{proof}
It suffices to observe that $V(u)(t)= \int_0^t(t-s)u(s)ds$ and apply the standard Gronwall lemma. 
 \end{proof}
As a consequence, we deduce that the 
preceding diagram  satisfies $(J)$. 
Indeed, 
if for some $\eta\ge 0$ one has
$$x(t)-y(t)=\eta\int_0^t \int_0^r \left[ f(s,x(s)) - f(s,y(s))\right]dsdr,$$
 then 
$$|x(t)-y(t)|\le \eta l \int_0^t \int_0^r |x(s)-y(s)|dsdr
$$
and  Lemma \ref{Gronwall de segundo orden} applies.

 \begin{lemma}
$\pi=\pik+\delta$ is onto.
\end{lemma}
\begin{proof}
By definition, $\pi(x)=t[ x(1)+x'(0) ]  + 2x(0)$. Given 
$y(t)=ta+b$, then $x(t)= t^2\left(a - \frac b2\right) + \frac b2$ satisfies $\pi(x)=y$. 
 \end{proof}

{In the previous setting,  consider the diagram 
$\mathcal{B}$ associated to (\ref{Dirichlet homogeneo 2 orden 1-bis})-(\ref{Dirichlet homogeneo 2 orden 2-bis}) and assume that the open bounded sets $U_1\subset C[0,T]$ and $U_2\subset \ker(L)$ have common core with respect to  $\mathcal{B}$. 
In the situation of Theorem \ref{teorema Lx=Nx}, we have that 
$$K(x)(t)=t[x(1)+x'(0)] +2x(0)+\int_0^t \int_0^r  f(s,x(s))dsdr,$$  
$$K_2(ta+b)=(\pi\circ\mu)(ta+b)$$
and it follows that
$$\dls_{LS}(I-K,U_1,0)=\dls_{B}(I-K_2,U_2,0).$$
Since we are dealing with a Dirichlet problem, it   is worthy comparing the latter operator with the  
shooting operator 
$\s:\R^n\to \R^n$ given by $\s(a):=x(1)$, where $x(t)$ is the 
unique solution of the equation such that $x(0)=0$ and $x'(0)=a$. For convenience, 
we may write $\s(a)=\Phi^1(1,0,a)$, where $\Phi(t,b,a)=(\Phi^1,\Phi^2)$ 
denotes the associated flow.  Moreover, we may define the (open) mapping $\varphi:\ker(L)\to\R^n$ given by $\varphi(ta+b):=a$. 

\begin{proposition} \label{deg-shoot} In the previous situation, 
$$\deg_B(\s,\varphi(U_2),0)=
   \deg_B(I-K_2,U_2,0). 
$$
\end{proposition}}
{\begin{proof}
Consider the subspace  $W:=\{ta:a\in\R^n\}\subset\ker(L)$ and define the function $g:\ker(L)\to W$ given by $g(ta+b)=t(a-\mathcal S(a))$. From the properties of the Brouwer degree, 
$$\deg_B(I-g,U_2,0)= \deg_B((I-g)|_W,{U_2\cap W},0).
$$
Because $(I-g)(ta)=t\mathcal S(a)$ and $\varphi|_W$ is an isomorphism, the latter degree is clearly equal to $\deg_B(\mathcal S,\varphi(U_2),0)$. On the other 
hand, observe that 
$$(I-K_2)(ta+b)= -(t\Phi^1(1,b,a) + b)  
$$
and hence
$$\deg_B(I-K_2,U_2,0)= (-1)^{2n}\deg_B(F,U_2,0) = \deg_B(F,U_2,0), 
$$
where $F(ta+b):= t\Phi^1(1,b,a) + b$. Thus, it suffices to verify that $F$ and $I-g$ are homotopic. To this end, suppose that
$$\lambda F(ta+b)+(1-\lambda)(t\mathcal S(a) + b)=0,
$$
then 
$$\lambda \Phi^1(1,a,b)+(1-\lambda)\mathcal S(a)=0 \qquad\hbox{and}\qquad b=0.
$$
Next, recall  that $\Phi^1(1,a,0)=\mathcal S(a)$; hence, we conclude that $\mathcal S(a)=0$ and it follows from the assumptions that $ta\notin \partial U_2$. 
\end{proof}}

Several other examples of fixed point operators for this problem shall be defined in section \ref{further}.

{
 \subsection{An elliptic problem with nonlocal conditions} \label{elliptic-sec}
 Let us consider the elliptic system  
 \begin{equation}\label{elliptic}
     \Delta u(x)= f(x,u(x))\qquad x\in \Omega
 \end{equation} 
with the nonlocal condition 
\begin{equation}
 \label{nonlocal} u|_{\partial\Omega} = \textit{constant},\qquad \int_{\partial\Omega}\frac{\partial u}{\partial \nu} dS=0.    
\end{equation}
Here, $\Omega\subset \R^d$ is a smooth bounded domain,  $f:\overline\Omega\times \R^n\to \R^n$ is smooth
and we set
$$X:= \{u\in C^{1,\alpha}(\overline\Omega): u|_{\partial\Omega} = \textit{constant}\},$$
$$Y= C^\alpha(\overline\Omega)$$
for some $\alpha\in (0,1)$;

\noindent $L:D\to Y$ given by $Lu:=\Delta u$, with $D:= X\cap C^{2,\alpha}(\overline\Omega); $
$$N(u)(x):=f(x,u(x))$$
and $V:C^\alpha(\overline\Omega)\to D$ defined by $V\varphi:=u$, the unique solution of the problem
$$\Delta u =\varphi,\qquad u|_{\partial \Omega}=0. 
$$
It is clear that $\ker(L)$ is the set of constant functions, which shall be identified with $\R^n$, and we may consider 
$$\pi_{\ker(L)}(u):=u|_{\partial\Omega}$$
and 
$$\delta(u):= \int_{\partial\Omega}\frac{\partial u}{\partial \nu} dS.
$$
In order to define a function $\mu$ as before, we shall assume the following monotonicity assumption
\begin{equation}
    \label{monot}
    \langle f(x,v)-f(x,u), v-u\rangle \ge -\kappa |u-v|^2
\end{equation}
where $\kappa < \lambda_1^{(D)}$, the first eigenvalue of $-\Delta$ for 
the Dirichlet conditions. Then $\mu:\R^n\to D$ is defined by $\mu(c):= u_c$, the unique solution of (\ref{elliptic}) such that $u|_{\partial\Omega}=c$. 
{The fact that $\mu$ is well defined is known, and can be regarded as one of those ``uniqueness implies existence" results. 
Indeed, if $u$ and $v$ are solutions of (\ref{elliptic}) such that $u=v$ on $\partial\Omega$, then
$$\langle\Delta(u-v),u-v\rangle|_x = \langle f(\cdot,u)-f(\cdot,v),u-v\rangle|_x \ge -\kappa |u(x)-v(x)|^2;
$$
thus, 
$$\int_{\Omega}|\nabla(u-v)(x)|^2\, dx \le \kappa \|u-v\|_{L^2}^2 
$$
and from the Poincar\'e inequality it follows that $u\equiv v$. Moreover, writing
$$\Delta u(x)=f(x,u(x)) - f(x,c) + f(x,c)
$$
and multiplying by $u(x)-c$
it is seen that
$$\|\nabla (u-c)\|_{L^2}^2 \le \kappa \|u-c\|_{L^2}^2 + k \|u-c\|_{L^2}
$$
for some constant $k$. This yields \textit{a priori} bounds for the solutions and the existence of $u_c$ follows from a standard 
truncation argument. Observe, furthermore, that the same computation shows that condition $(J)$ is satisfied: if $u-v= \eta V(N(u)-N(v))$ with $\eta\ge 0$, then $u=v=0$ on $\partial \Omega$ and  $$\Delta(u-v)(x)=\eta (f(x,u(x))-f(x,v(x)).$$ As before, it is deduced that $u\equiv v$.
}

In this setting, one has
$$\pi(u)=\pi_{\ker(L)}(u) + \delta(u)= u|_{\partial\Omega} + \int_{\partial\Omega}\frac{\partial u}{\partial \nu} dS,
$$
which is clearly onto,
and
$$K(u)= \pi(u)+ (V\circ N)(u)= v,$$
where $v$ is the unique solution of (\ref{elliptic}) such that 
$$v|_{\partial\Omega} = u|_{\partial\Omega} + \int_{\partial\Omega}\frac{\partial u}{\partial \nu} dS.
$$
It is readily seen  that $K$ is a fixed point operator. Moreover,  we may also compute $K_1= \pi \circ \mu:\R^n\to \R^n$ and $K_2=\mu \circ \pi:X\to X$ explicitly as
$$K_1(c)= c + \int_{\partial\Omega}\frac{\partial u_c}{\partial\nu}dS, 
$$
$$K_2(u)=u_{\pi(u)},
$$
that is, $K_2(u)$ is  the unique solution  of (\ref{elliptic}) whose value over the boundary of $\Omega$ is equal to
$ u|_{\partial\Omega} + \int_{\partial\Omega}\frac{\partial u}{\partial\nu}dS$.
 From the previous results, we deduce: 
\begin{proposition} Assume that $U_1\subset \R^n$ and $U_2\subset X$ are open bounded sets with  common core, then 
$$\deg_{LS}(I-K_2,U_2,0)=\deg_{LS}(I-K,U_2,0)=\deg_{B}(I-K_1,U_1,0).$$
\end{proposition}}  
{As in the preceding examples, several other fixed point operators can be defined; in particular, the one that arises on the 
Lyapunov-Schmidt reduction.  
In the first place, notice that $\pi$ is a projector, so  a fixed point operator $K_\gamma$ can be  defined in the same way  as before. Thus, taking $\gamma(u)=\overline u$ yields the operator
$$\overline u +|\Omega| \overline {Nu} + V(N(u)) - \overline{V(N(u))}
$$
which, in turn, is homotopic to 
$$K_3(u)=\overline u +|\Omega| \overline {Nu} + V(N(u)-\overline {Nu}) - \overline{V(N(u)-\overline {Nu})}.
$$
Finally, observe again that the latter operator is conjugated to the one  defined by the same formula, after replacing $D$ by 
$$\textrm{dom}(L)= C^{2,\alpha}(\overline\Omega)\cap \delta^{-1}(0),$$ 
which corresponds 
 to the Lyapunov-Schmidt 
 decomposition and,  as in section \ref{period}, may be called  $K_5$. 
 In this case, if one considers the operators 
 $K^\eta:u\mapsto v$, defined by 
 $$\Delta v(x) + \eta v(x)=f(x,u(x)) + \eta u(x) 
 $$
 where $\eta$ is not an eigenvalue of $-\Delta$ for the conditions (\ref{nonlocal}), then $K^\eta$ is linearly homotopic to 
 $K_5$ when $\eta>0$. Interestingly, the interaction of $\eta$ with the spectrum of $-\Delta$ is not relevant in this case: an intuitive explanation of this fact 
 shall be given in section \ref{further}.
 }
  
\subsection{Delay differential equations}
Let us consider the  problem of finding periodic solutions to the delay equation
\begin{align}
x'(t)&=f(t,x(t),x(t-\tau)) \quad  t\geq 0 \label{DDE}\\
x_0&=x_T\label{CPD}
\end{align} where $f:\R\times \R^n\times \R^n \to \R^n$ is continuous and Lipschitz in the last two variables and $T$-periodic in $t$, with $T\ge \tau$. Here, as usual, we employ the notation $x_t\in C[-\tau,0]$ for the function defined by $x_t(s):=x(s+t)$; thus, a solution of (\ref{DDE}) satisfying 
(\ref{CPD}) 
is extended  
to a periodic solution of (\ref{DDE}).  
In this case, we may consider the diagram $\mathcal{B}_{\tau}$ as follows

 \begin{figure}[h]
    \begin{center}
\includegraphics[scale=.45]{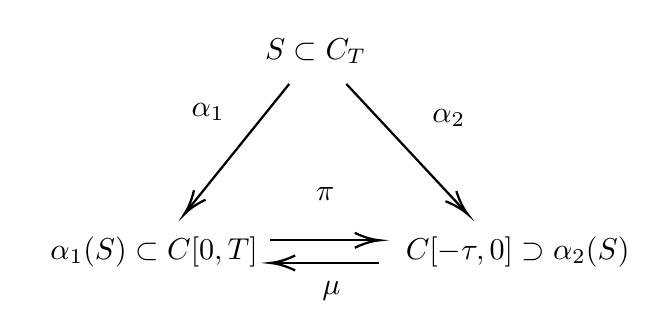}
\end{center}
\end{figure}

\noindent where:
\begin{itemize}
\item $C_T:=\{x\in C(\R,\R^n):x(t+T)=x(t)\}$. 
    \item $S\subset C_T$ is the set of  $T$-periodic solutions of (\ref{DDE}).      \item $\alpha_1$ denotes the  restriction to the interval $[0,T]$.
    \item $\alpha_2$ denotes the  restriction to the interval $[-\tau,0]$, that is $\alpha_2(x):=x_0$.
    \item $\pi(x):=x_T$
    \item  $\tilde{\mu}(y)$ is the solution of (\ref{DDE}) with initial condition  $x_0=y$, defined in $[-\tau,T]$ and  $\mu(y):= \tilde{\mu}(y)\vert_{[0,T]}$.
\end{itemize}

 It is easy to verify that   $\mu$ is compact and $\pi$ is linear and continuous. Furthermore, observe that it has a right inverse
 $i:C[-\tau,0]\to C[0,T]$ given by  $$i(y)(t)=\mathbb{I}_{[T-\tau,T]}(t)y(t-T)+\mathbb{I}_{[0,T-\tau)}(t)y(-\tau)$$
  where $\mathbb I$ stands for the indicator function. 
 
 We claim that the diagram, restricted to  $S$, $\alpha_1(S)$ and $\alpha_2(S)$, is commutative. 
 Indeed, for $x\in S$ it is seen that  $(\pi\circ\alpha_1)(x)=x_T$. Because
 $x\in C_T$, it follows that $x_T=x_0$ and hence $(\pi\circ\alpha_1)(x)=\alpha_2(x)$. On the other hand,  $(\mu\circ\alpha_2)(x)=\mu(x_0)$, which defines the unique solution of (\ref{DDE}) with initial condition $x_0$, restricted to   $[0,T]$. In other words,  $(\mu\circ\alpha_2)(x)=x|_{[0,T]}=\alpha_1(x)$.

Next, consider the following diagram $\mathcal{K}_\tau$  given by 
 \begin{figure}[h]
    \begin{center}
\includegraphics[scale=.45]{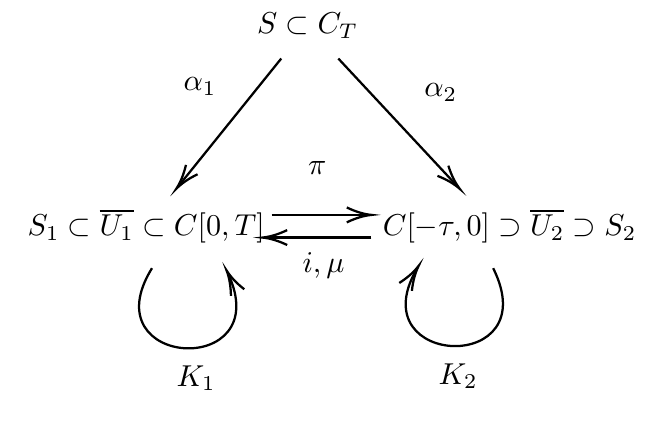}
\end{center}
\end{figure}

\noindent where:
\begin{itemize}
    \item $i$ is defined as before.
    \item $K_1=\mu\circ\pi$ and $K_2=\mu\circ \pi$.
    \item $U_i$ have common core with respect to $\mathcal{B}_\tau$.
\end{itemize}
\begin{lemma}
$\mathcal{K}_\tau$ is a Krasnoselskii  diagram. 
\end{lemma}
\begin{proof}
We already know that the linear continuous function  $i$ is a right inverse of  $\pi$; thus, it only remains to verify that  $K_1$ and $K_2$ are fixed point operators. 
Let $x\in S$, then  $$K_1(\alpha_1(x))=\mu(\pi \circ \alpha_1(x))=\mu(x_T)=\mu(x_0)=x\vert_{[0,T]}=\alpha_1(x).$$ 
Conversely, let 
 $\hat x\in \fp(K_1)$, that is, $\hat x=(\mu\circ \pi)(\hat x)=\mu(\hat x_T)$. Evaluation at $t=0$ yields
 $$\hat x(0)= \mu(\hat x_T)(0)=\hat x_T(0)= \hat x(T). 
 $$
Let us consider now the function $x:[-\tau,T]$
 defined by ${x}\vert_{[-\tau,0]}=\hat x_T$ and ${x}\vert_{(0,T]}=\hat x$.
 Since $\hat x(0)=\hat x(T)$, it is deduced that $x$ is  continuous. Moreover, over $[0,T]$ we have that $x=\hat x=\mu(\hat x_T)$, 
 which is the restriction of the unique  solution $y$ of (\ref{DDE}) satisfying $y_0=\hat x_T=x_0$. 
We conclude that $x=y$ and, because $x_0=x_T$, it follows that $x$ can be extended periodically to a solution of (\ref{DDE}), that is, $x\in S$ and $\alpha_1(x)=\hat x$. 

Concerning $K_2$, we proceed in an analogous  way. Let  $x \in S$, then $$K_2(\alpha_2(x))=(\pi\circ\mu)(x_T)=\pi(\mu(x_0))=\pi(x\vert_{[0,T]})=x_T=x_0=\alpha_2(x).$$
Conversely, let  $y\in \fp(K_2)$, that is $y=(\pi\circ\mu)(y)$. Let $\tilde y$ be the unique solution of (\ref{DDE}) such that $\tilde y_0=y$, 
then $y=\pi(\mu(y))=\tilde y_T$. This implies that 
$\tilde y$ is extended periodically to an element  $\tilde y\in S$ and $y=\alpha_2(\tilde y)$. 
\end{proof}
As in section \ref{period}, an affinity principle is obtained between $K_1$ and $K_2$. Note that $K_2$ corresponds again to the Poincar\'e operator $P$ although, in contrast with the non-delayed case, now $P$ is defined over the infinite-dimensional space $C[-\tau,0]$. 
 \begin{corollary}
$$\dls_{LS}(I-K_1,{U_1},0)=\dls_{LS}(I-P,{U_2},0)$$
\end{corollary}

Next, let us analyze the operator $K:C[0,T]\to C[0,T]$ given by $$K(x)(t)=x(T)+\int_0^t f(u, x(u), x(r(u)))du,$$ where $r(u)=:\begin{cases}
       u-\tau +T \quad & \ \text{if} \;\; 0\leq u<\tau \\
       u-\tau \quad & \ \text{if} \;\; \tau\leq u \leq T.  \\
     \end{cases}$
     
Notice that, despite the fact that $r$ is not continuous, the operator $K$ is well defined and compact.

\begin{lemma}
$K$ is a fixed point operator for  $\mathcal{B}_\tau$
\end{lemma}

\begin{proof} Let 
 $x\in S$, then  $x'(t)=f(t,x(t),x(t-\tau))$ for all $t\in \mathbb{R}$, that is
$$x(t)=x(0)+\int_0^t f(u,x(u),x(u-\tau))du
$$
and, because $x\in C_T$, it follows that 
$x|_{[0,T]}=\alpha_1(x)$ is a fixed point of $K$.

Conversely, if  
  $\hat x\in \fp(K)$, then clearly $\hat x(T)=\hat x(0)$. Moreover, if 
    $${x}(t):=\mathbb{I}_{[-\tau,0]}(t)\hat x_T(t)+\mathbb{I}_{[0,T]}(t)\hat x(t),$$
then 
$x_0=\hat x_T=x_T$ and $x(r(u))=x(u-\tau)$ for $u\in [0,T]$. Because $\hat x=K(\hat x)$, for $t\in [0,T]$  it follows that 
$$
x(t)=\hat x(t)= \hat x(T) +\int_0^t f(u,\hat x(u),\hat x(r(u))du =
x(T) +\int_0^t f(u, x(u), x(u-\tau))du.
$$ 
Thus, $x$ can be extended to a $T$-periodic solution of (\ref{DDE}) and hence $\hat x\in \alpha_1(S)$. 
\end{proof}

\begin{theorem}
$$\dls_{LS}(I-K,{U_1},0)=\dls_{LS}(I-K_2,{U_2},0)$$
\end{theorem}

\begin{proof}
From the previous corollary, it suffices to verify that  $\dls_{LS}(I-K,{U_1},0)=\dls_{LS}(I-K_1,{U_1},0)$.
To this end, consider again  the linear  homotopy  $H_{\lambda}=\lambda K + (1-\lambda)K_1$ with  $\lambda\in [0,1]$. Suppose that $\hat x\in \overline U_1$ is a fixed point of $H_\lambda$, then 
$$\hat x(t)=\lambda \left( \hat x(T)+\int_0^t f(u, \hat x(u),\hat x(r(u)))du \right) +  (1-\lambda)  (\mu \circ \pi)(\hat x)(t).  
$$
Evaluation at  $t=0$ yields $\hat x(0)=\hat x(T)$. 
As in the previous theorem, let
$x:=\hat x_T\mathbb I_{[-\tau,0)} + \hat x\mathbb I_{[0,T]}\in C[-\tau, T]$, then 
\begin{equation}\label{hom-delay}
{x}(t)=\lambda \left( {x}(T)+\int_0^t f(u, {x}(u), {x}(u-\tau)du \right) +  (1-\lambda) \tilde{\mu}(  {x}_T)(t)    
\end{equation} 
for all $t\in [0,T]$ and ${x}_0=x_T 
$.  Now consider ${y}=\tilde{\mu}({x}_T)$, then 
\begin{align}
{y}'(t)&=f(t,{y}(t),{y}(t-\tau)),  \;\;\;\;  \forall t\in [0,T]\label{tilde y derivada}\\
{y}_0&={x}_0\label{tilde y_0}
\end{align}
Moreover, from (\ref{hom-delay}), we  obtain
\begin{equation}\label{eq homotopia delay derivado}
    {x}'(t)=\lambda  f(t,{x}(t),{x}(t-\tau))+(1-\lambda)f(t,{y}(t),{y}(t-\tau)).
\end{equation}
Because $y$ is a solution of the latter equation and $y_0=x_0$, 
it follows from the uniqueness that $y=x$.  

\begin{obs}
As before, using Gronwall's lemma it follows, more generally, that condition $(J)$ is fulfilled. 
\end{obs}

  As a conclusion, we deduce that $x$, extended periodically to $\mathbb R$, belongs to  $S$ and $\alpha_1(x)=\hat x$. Moreover,  since $\alpha_1(S)\cap \partial(U_1)=\emptyset$ we conclude that  $\hat x\in U_1$; thus, $H_\lambda$ has no fixed points on 
 $\partial U_1$ and the result follows.
\end{proof}

\subsubsection{Other fixed point operators in $C[0,T]$ and $C_T$}


For notational convenience, let us define  
 $$Vx(t)=\int_0^tx(s)ds$$
$$N_r(x)(t):=f(t,x(t),x(r(t))$$ 
with $r$ as before
and observe that, although $N_r(x)$ is not a continuous function, the operator $K_6:C[0,T]\to C[0,T]$ given by
$$
K_6(x)(t):=\overline x + T\overline{N_r(x)} + V(N_r(x))(t) - 
\overline {V(N_r(x))}
$$
is well defined and compact. 
In the same way, we may define the operators
{$K_7:C[0,T]\to C[0,T]$ and $K_8:C_T\to C_T$ given by
$$K_7(x)(t):= \overline x + T\overline{N_r(x)} + 
V(N_r(x)-\overline{N_r(x)})(t) - \overline{V(N_r(x)-\overline{N_r(x)})}
$$
$$K_8(x)(t):= \overline x + T\overline{N(x)} + 
V(N(x)-\overline{N(x)})(t) - \overline{V(N(x)-\overline{N(x)})}.
$$
}
{The operators are  analogous, respectively, to $K_4$,  $K_3$ and $K_5$ in section \ref{period}. The proof of the following proposition is left for the reader.}

\begin{proposition}
$K_6, K_7$ and $K_8$ are fixed point operators and
$$\deg_{LS}(I-K_6,U,0)=\deg_{LS}(I-K_7,U,0)= \deg_{LS}(I-K_8,\tilde U,0)
$$
provided that the open bounded sets 
$U\subset C[0,T]$ and $\tilde U\subset C_T$ have common core. 
\end{proposition}

\section{Final comments and open questions}
\label{further}

\begin{enumerate}
    \item In section \ref{period}, 
    an alternative proof of the case $\eta<0$ in  Proposition \ref{eta-period} may be performed as follows. 
    In the first place, observe that an analogue of Corollary \ref{k3} is obtained for $\hat K_3$ if we take, in the original diagram,
$$\hat \pi_{\ker(L)}(x)= x(T), \qquad \hat\delta(x)=x(0)- x(T),$$
$$
\hat V(x)(t)=-\int_t^T x(s)ds, \qquad \hat\mu(x_T)(t)= \hat \Phi(t, x_T),
$$
where $\hat \Phi(t,A)$ is defined as the unique solution of the equation such that $\hat \Phi(T,A)=A$. 
Then $\hat \pi(x)=x(0)$ and
$$\hat K(x)(t) = x(0) - \int_t^T f(s,x(s))ds, $$
 while
$$\hat K_1(x)(t)= \hat  \Phi(t,x(0)).
$$
Thus, the abstract result implies that, if the   
common core assumption is satisfied, then the degree of   $I-\hat K_1$ coincides with the degree of $I-\hat K_2$, where
$$\hat K_2(A)= (\hat \mu\circ\hat \pi) =\hat \Phi(0,A).$$
In other words, $\hat K_2$ can be identified with the  inverse $\hat P$ of the  Poincar\'e map; thus, the degree computation follows  from the next lemma:


\begin{lemma}
Let 
$\Omega\subset \R^n$ be open and bounded such that $P$ has no fixed points on $\partial\Omega$, then 
$$\deg_B(I-\hat P,P(\Omega),0)= (-1)^n\deg_B(I- P,\Omega,0).$$
\end{lemma} 
\begin{proof}
 It suffices to observe that   $P$ and 
 $\hat P$ have the same fixed points and, due to the properties of the Brouwer degree, we may assume that  
 $0$ 
is a regular value of $I-P$. 
Because the degree over small balls coincides with the degree of its linearization, we only need to prove that $$\deg_B(I-P_L,B_r(0),0)=(-1)^n\deg_B(I-\hat P_L,P_L(B_r(0)), 0),$$
where $P_L$ is the Poincar\'e map associated to  a linear system $x'(t)=B(t)x(t)$, where $B$ is a $T$-periodic matrix such that $P_L$ has no nontrivial fixed points. Let $M(t)$ be the fundamental matrix of the system $x'=Bx$ such that $M(0)=I$, 
then $P_L(x_0)=M(T)x_0$, so the degree of $I- P_L$ is simply computed by
$$\deg_B(I-P_L,B_r(0),0)= sgn(\det(I-M(T))).$$
On the other hand, 
$$\deg_B(I-\hat P_L,P_L(B_r(0)), 0)= 
sgn(\det(I-M(T)^{-1})) $$
$$= (-1)^nsgn\det(M(T)^{-1})sgn(\det(I-M(T))).$$
Thus, the result follows from the fact that $M(t)$ 
is invertible for all $t$ and, by continuity,
$$sgn(\det(M(T))= sgn(\det(M(0))=1.
$$
\end{proof}

Also, it is worth noticing
that, in the context of the coincidence degree theory, it is well known that
if the assumption
\begin{equation}
\label{lyap-schm}  x'\ne \lambda N_T(x),\qquad x\in \partial\tilde U,  \, \lambda\in (0,1]  
\end{equation}
holds, then  
$$\deg_{LS}(I- K_5,\tilde U,0)=\deg_B(\phi,  \tilde U\cap\R^n,0)
$$
and 
$$\deg_{LS}(I-\hat K_5,\tilde U,0)=\deg_B(-\phi,  \tilde U\cap\R^n,0)
$$
where $\phi:\R^n\to \R^n$ is defined by 
$$\phi(x):= -TN_T(x). 
$$
In other words, when (\ref{lyap-schm}) is satisfied, 
taking $U, \tilde U\subset C_T$ as in the previous proposition it is seen in a direct way that
$$\deg_{LS}(I-K_5,\tilde U,0)= 
(-1)^n \deg_{LS}(I-\hat K_5, U,0).
$$

\item Concerning the Dirichlet problem in section \ref{dirich}, several comments can be made. To begin, observe that the   explicit form of the operator $K_2$ is 
$$K_2(ta+b)= (\pi\circ\mu)(ta+b)= t[x(1)+x'(0)] + 2x(0) = t[x(1)+a] + 2b,
$$
where $x=\mu(ta+b)=\Phi^1(\cdot,b,a)$.  
In particular,  if $K_2(ta+b)=ta+b$, then $b=0$: in other words, 
the fixed points of $K_2$ belong to the subspace $W =\{ ta:a\in \R^n\}\subset\ker(L)$.    
Notice, furthermore, that $K_2(W)\subset W$. 
This yields a direct computation of  
degree as follows. According to the standard properties 
of the Brouwer degree, may assume that $0$ is a regular value of $I-K_2$, 
then 
$$\deg_B(I-K_2,U_2,0)=\sum_{ta\in S_2} sgn(\det(I-DK_2(ta))) 
$$
where $S_2:=\{ ta\in U_2:K_2(ta)=ta\}$. A simple computation shows that $DK_2(ta)$ can be identified with the $\R^{2n\times 2n}$ matrix
given by 
$$\left(\begin{array}{cc}
     D_a\Phi^1(1,0,a) + I_n &  D_b\Phi^1(1,0,a)\\
    0 & 2I_n
\end{array}
\right)
$$
where  $I_n\in \R^{n\times n}$ denotes the identity matrix.
 This implies that $$
sgn(\det(I-DK_2(ta))) = sgn\left(\det\left(D_a\Phi^1(1,0,a) \right)\right).$$
Moreover, taking into account   the isomorphism $\varphi|_W:ta\mapsto a $
it is clear that the zeros of 
$\s$ in $\varphi(\overline U_2)$ coincide with the elements of $\varphi(S_2)$. Because
$D\s(a)=D_a\Phi^1(1,0,a)
$, we conclude that 
  $$\deg(\s,\varphi(U_2),0)=
  \sum_{a\in \varphi(S_2)} sgn\left(\det\left(D_a\Phi^1(1,0,a) \right)\right)
$$
and so the result
is deduced.

\item Still regarding section \ref{dirich}, we may define a variety of fixed point operators as follows. Firstly,
 {inspired by Proposition \ref{conjugacion},
we may define an operator $\tilde K$ as a conjugate of the function $g$ defined in the proof of Proposition \ref{deg-shoot}. With this in mind, let us define $\tilde i:W\to C^1[0,1]$ given by $\tilde i(ta):=\mu(ta)=\Phi^1(\cdot,0,a)$ and $\tilde\pi:C^1[0,1]\to W$ given by 
$\tilde \pi(x):=tx'(0)$. Then 
 $\tilde\pi\circ \tilde i=\id_{W}$ and, according to the definition  $g(ta):=t(a-\mathcal S(a))$, we may compute  $\tilde K=\tilde i\circ g\circ \tilde \pi$ to obtain:
$$\tilde K(x)= \Phi^1(\cdot,0,x'(0)-\mathcal S(x'(0))).
$$
 It is interesting to observe that $\tilde K$ and $K:=\pi + V\circ N$ satisfy the affinity principle, although they are not linearly homotopic. The same happens with the operator $K_1=\mu\circ \pi$, which reads
$$K_1(x)=\mu[t(x(1)+x'(0)) +2x(0)].$$  
Indeed, from the equality  $x=\lambda K_1(x) + (1-\lambda)\tilde K(x)$ one obtains  $x(0)=2\lambda x(0)$, but the conclusion $x(0)=0$ does not necessarily hold when  $\lambda=\frac 12$.} 

 {More generally, a family of fixed point operators related with all the preceding ones can be obtained in the following way. Consider the mapping $\delta_\theta:=\theta\circ\delta$, where $\theta:\ker(L)\to \ker(L)$ is an arbitrary orientation-preserving isomorphism. It is well known that   $\theta$ and the identity belong to the same connected component of the linear group $GL(\ker(L))$ and, thus, we may define $\pi_\lambda:=  \pi_{\ker(L)} + h(\lambda) \circ\delta$, where $h:[0,1]\to GL(\ker(L))$ is continuous with $h(0)=I$ and $h(1)=\theta$. Observe that $\pi_\lambda$ may not be necessarily onto but, however, the corresponding $K_\lambda$ is a fixed point operator: if $x=K_\lambda(x)$, then $L(x)=N(x)$; moreover, applying   
$\pi_{\ker(L)}(x)$ it follows that $h(\lambda)(\delta(x))=0$ and, consequently, $\delta(x)=0$. 
If one takes,  
instead, an element $\theta\in GL(\ker(L))$ that does not preserve orientation, then it belongs to the same component of 
$\hat \delta(x)=-tx(1)+x(0)$ and it follows that $\deg(I-K_2(\theta),U_2,0)=(-1)^n\deg(\mathcal S,\varphi(U_2),0).$ }

{It is clear that one may also replace the projector $\pi_{\ker(L)}$ by another one, although the application of the conjugacy lemma relies on the existence of $\mu$. On the other hand, a family of fixed point operators of different nature is given by $K^\eta:x\mapsto y$, where $y$ is the unique solution of the problem  
$$y''(t)+\eta y(t) = f(t,x(t)) + \eta x(t),\qquad y(0)=y(1)=0$$ and 
$\eta$ is not an eigenvalue of $-L$ for the Dirichlet condition. 
Here, the degree of $I-K^\eta$ will depend on the position of $\eta$ with respect to the spectrum of $-L$; this study shall be the subject of a forthcoming paper. }

\item Analogously, an operator $K^\eta$ was also  defined in 
section \ref{elliptic-sec}, assuming that $-\eta$ does not belong to the spectrum  of $-\Delta$ under the nonlocal conditions (\ref{nonlocal}).  
As mentioned, in this case
the degree of $I-K^\eta$ depends only on the sign of $\eta$: 
more precisely, following the ideas of section \ref{period}
it can be proven that  
the degree of $I-K^\eta$ coincides with the degree of $I-K_5$, multiplied by  
$sgn(\eta)^n$. For $\eta>0$, we already observed that $K^\eta$ is homotopic to $K_5$; when $\eta<0$, an intuitive explanation reads as follows. For simplicity, assume that   $d=1$ (ODE case), which corresponds, when $\Omega=(0,2\pi)$, to the second order $2\pi$-periodic problem. 
  With this in mind, consider the  linear autonomous problem
 $$u''(t)=Au(t)
 $$
 with $A\in \R^{n\times n}$ and suppose there are no nontrivial $T$-periodic solutions. If one defines as before the mapping $u\mapsto v$ by solving the equation
 $$v'' +  \eta v = Au +\eta u,
 $$
 then $(I-K^\eta)(u)=u-v$ is computed in terms of the Fourier expansion of $u$ given by
 $$u= a_0 + \sum_{n\in\mathbb N}[\cos(nt) a_n +\sin(nt)b_n]
 $$
 as 
 $$(I-K^\eta)(u)= \left(I-\frac A\eta\right)a_0 + 
 \sum_{n\in\mathbb N} \left(\frac{\eta I - A}{\eta-n^2}
 \right)[\cos(nt) a_n +\sin(nt)b_n].
 $$
 This implies that the subspaces  $S_n:=\textrm{span}\{\cos(nt), \sin(nt)\}$ are invariant and, furthermore, the associated matrix of $(I-K^\eta)|_{S_n}$ is 
 $$M_n=\left(\begin{array}{cc}
 \frac{\eta I - A}{\eta-n^2}      & 0 \\
     0 & \frac{\eta I - A}{\eta-n^2} 
 \end{array}
 \right)
 $$
 which satisfies $sgn(\det(M_n))>0$. In other words,  the degree computation depends only on the sign of $\det\left(I-\frac A\eta\right)$.

\end{enumerate}

\subsection*{Acknowledgement}
The authors would like to thank Prof. Rafael Ortega for the reading of a first version of the manuscript and his fruitful comments.  
This work was  partially supported by  
projects UBACyT
20020190100039BA  and PIP 11220130100006CO CONICET.


\begin{thebibliography}{1}



\bibitem{BK} 
N. Bobylev and 
 M. A. Krasnoselskii, \textit{Principles of affinity in nonlinear problems}.  
 Journal of Mathematical Sciences 83 No. 4 (1997), 485--521.



\bibitem{GM}
 R. Gaines and J. Mawhin, \textit{
Coincidence degree and nonlinear differential equations}. Springer, 1977. 


\bibitem{K} 

 M. A. Krasnoselskii,
\emph{The operator of translation along the trajectories of differential equations}, 
 American Mathematical Society, Providence RI, 1968.


\bibitem{krasno} M. A. Krasnoselskii and P. P. Zabreiko, \textit{Geometrical methods of nonlinear analysis}. Springer-Verlag, Berlin, 1984.


\bibitem{maw} J. Mawhin, \textit{Topological degree methods in nonlinear boundary value problems}.
CMBS Conf. Math. No. 40, Amer. Math. Soc., Providence, 1979.




\end{thebibliography}
\end{document}